\documentclass[a4paper,11pt,reqno,twoside]{article}
\usepackage{amssymb}
\usepackage[frenchb,english]{babel}
\usepackage{amscd}
\usepackage{amsmath,amsthm,amsfonts,amssymb,graphicx}
\usepackage[colorlinks,linkcolor=blue,anchorcolor=blue,citecolor=green]{hyperref}
\usepackage{mathrsfs}
\usepackage{fancyhdr}
\usepackage{subfigure}
\usepackage{bm}
\usepackage{multicol}
\usepackage{picins}
\usepackage{abstract}

\makeatother

\begin{document}
\allowdisplaybreaks
\newtheorem{theorem}{Theorem}[subsection]
\newtheorem{theorem1}{Theorem}[section]
\newtheorem{proposition}{Proposition}[section]
\newtheorem{lemma}{Lemma}[section]
\numberwithin{equation}{section}
\theoremstyle{plain}
\newtheorem{Definition}{Definition}[section]
\newtheorem{Proposition}{Proposition}[section]
\newtheorem{Property}{Property}[section]
\newtheorem{Theorem}{Theorem}[section]
\newtheorem{Lemma}[Theorem]{\hspace{0em}\bf{Lemma}}
\newtheorem{Corollary}[Theorem]{Corollary}
\newtheorem{Remark}{Remark}[section]
\newtheorem{Example}{Example}[section]

\setlength{\oddsidemargin}{ 1cm}  
\setlength{\evensidemargin}{\oddsidemargin}
\setlength{\textwidth}{13.50cm}
\vspace{-.8cm}

\noindent{\Large Rigidity of proper holomorphic mappings between generalized Fock-Bargmann-Hartogs domains}\\\\

\noindent\text{Enchao Bi$^{1}$$^{*}$ \; \&  \; Zhenhan Tu$^{2}$}\\

\noindent\small {${}^1$School of Mathematics and Statistics, Qingdao
University, Qingdao, Shandong 266071, P.R. China} \\
\noindent\small {${}^2$School of Mathematics and Statistics, Wuhan
University, Wuhan, Hubei 430072, P.R. China} \\

\noindent\text{Email: bienchao@whu.edu.cn (E. Bi),\;
zhhtu.math@whu.edu.cn (Z. Tu)}
\renewcommand{\thefootnote}{{}}
\footnote{\hskip -16pt {$^{*}$Corresponding author. \\ } }
\\

\normalsize \noindent\textbf{Abstract}\quad  A generalized Fock-Bargmann-Hartogs domain
$D_n^{\mathbf{m},\mathbf{p}}$ is defined as a domain fibered over
$\mathbb{C}^{n}$ with the fiber over $z\in \mathbb{C}^{n}$ being a generalized complex ellipsoid $\Sigma_z({\mathbf{m},\mathbf{p}})$. In general, a generalized Fock-Bargmann-Hartogs domain is an unbounded non-hyperbolic domains without smooth boundary. The main contribution of this paper is
as follows. By using the explicit formula of Bergman kernels of the generalized Fock-Bargmann-Hartogs domains, we obtain the rigidity results of proper holomorphic mappings between two equidimensional  generalized
Fock-Bargmann-Hartogs domains. We therefore exhibit an example of unbounded weakly pseudoconvex domains on which the rigidity results of proper holomorphic mappings can be built.\\

\noindent \textbf{Key words:} Automorphism groups,  Bergman kernels, Generalized Fock-Bargmann-Hartogs domains, Proper holomorphic mappings\\

\noindent \textbf{Mathematics Subject Classification (2010):} 32A07\textperiodcentered\, 32A25\textperiodcentered\, 32H35\textperiodcentered\, 32M05 \\

\setlength{\oddsidemargin}{-.5cm}  
\setlength{\evensidemargin}{\oddsidemargin}
\pagenumbering{arabic}
\renewcommand{\theequation}
{\arabic{section}.\arabic{equation}}

\section{Introduction}
A holomorphic map $F:\Omega_{1}\rightarrow \Omega_{2}$  between two domains $\Omega_{1},\; \Omega_{2}$ in $ \mathbb{C}^{n}$  is said to be proper if $F^{-1}(K)$ is compact in $\Omega_{1}$ for every compact subset $K\subset \Omega_{2}$. In particular,  an  automorphism $F:\Omega\rightarrow \Omega$ of a domain $\Omega$ in $ \mathbb{C}^{n}$  is a proper holomorphic  mapping of $\Omega$ into  $\Omega$.
There are many works  about proper holomorphic  mappings between various bounded domains with some requirements of the boundary (e.g., Bedford-Bell \cite{Bedf}, Diederich-Fornaess \cite{Diederich},
 Dini-Primicerio \cite{Dini} and  Tu-Wang \cite{T-W.Math Ann}). However, very little seems to be known about proper holomorphic mapping between the unbounded weakly pseudoconvex domains. There are also
some works about automorphism groups of hyperbolic domains (e.g., Isaev \cite{Isa}, Isaev-Krantz  \cite{IK}  and Kim-Verdiani  \cite{KV} ). In this paper, we mainly focus our attention on some unbounded non-hyperbolic weakly pseudoconvex domains.

The Fock-Bargmann-Hartogs domain $D_{n,m}(\mu)$ is defined by
$$D_{n,m}(\mu)=\{(z,w)\in \mathbb{C}^{n}\times \mathbb{C}^{m}: \left\lVert w \right\rVert^{2} <
e^{-\mu{\left\lVert z\right\rVert}^{2}}\} \;\; \mbox{for} \;\;  \mu>0,$$ where
$\|\cdot\|$ is the standard Hermitian norm. The
Fock-Bargmann-Hartogs domains $D_{n,m}(\mu)$ are strongly
pseudoconvex domains in $\mathbb{C}^{n+m}$ with smooth real-analytic boundary.  We note that each
$D_{n,m}(\mu)$ contains $\{(z, 0)\in
\mathbb{C}^n\times\mathbb{C}^m\} \cong \mathbb{C}^n$. Thus each
$D_{n,m}(\mu)$ is not hyperbolic in the sense of Kobayashi and
$D_{n,m}(\mu)$ can not be biholomorphic to any bounded domain in
$\mathbb{C}^{n+m}$. Therefore, each Fock-Bargmann-Hartogs domain
$D_{n,m}(\mu)$ is an unbounded non-hyperbolic domain in
$\mathbb{C}^{n+m}.$

In 2013, Yamamori \cite{Y} gave an explicit formula for the
Bergman kernels of the Fock-Bargmann-Hartogs domains in terms of the
polylogarithm functions. In 2014, by checking that the Bergman
kernel ensures revised the Cartan's theorem, Kim-Ninh-Yamamori
\cite{Kim} determined the automorphism group of the
Fock-Bargmann-Hartogs domains as follows.

\begin{Theorem} [Kim-Ninh-Yamamori \cite{Kim}] {\it The
automorphism group ${\rm Aut}(D_{n,m}(\mu))$ is exactly the group
generated by all automorphisms of  $D_{n,m}(\mu)$ as follows:
\begin{equation*}
\begin{array}{l}
\varphi_{U}:(z,w)\longmapsto(Uz,w),\quad U\in \mathcal{U}(n); \\\\
\varphi_{U^{'}}:(z,w)\longmapsto(z,U^{'}w),\quad U^{'}\in \mathcal{U}(m);\\\\
\varphi_{v}:(z,w)\longmapsto(z+v,e^{-\mu \langle z,v
\rangle-\frac{\mu}{2} {\left\lVert v\right\rVert}^{2}}w),\quad (v\in
\mathbb{C}^{n}),
\end{array}
\end{equation*}
where $\mathcal{U}(k)$ is the unitary group of degree $k,$ and
$\langle \cdot,\cdot \rangle$ is the standard Hermitian inner
product on $\mathbb{C}^{n}$.}
\end{Theorem}

Recently, Tu-Wang \cite{T.W} has established the rigidity of the proper holomorphic mappings between two equidimensional Fock-Bargmann-Hartogs domains as follows.

\begin{Theorem}[Tu-Wang \cite{T.W}]
{If $D_{n,m}(\mu)$ and $D_{n',m'}(\mu')$ are two equidimensional
Fock-Bargmann-Hartogs domains with $m\geq 2$ and $f$ is a proper
holomorphic mapping of $D_{n,m}(\mu)$ into $D_{n',m'}(\mu')$, then
$f$ is a biholomorphism  between $D_{n,m}(\mu)$ and
$D_{n',m'}(\mu')$.}
\end{Theorem}

A generalized complex ellipsoid  (also called generalized
pseudoellipsoid) is a domain of the form
\begin{eqnarray*}
\begin{aligned}
\Sigma(\mathbf{n};\mathbf{p})&=\{(\zeta_1,\cdots,\zeta_r)\in
\mathbb{C}^{n_1}\times\cdots\times\mathbb{C}^{n_r}:
\sum_{k=1}^{r}\|\zeta_k\|^{2p_k}<1 \},
\end{aligned}
\end{eqnarray*}
where $\mathbf{n}=(n_1,\cdots,n_r)\in \mathbb{N}^{r}$ and
$\mathbf{p}=(p_1,\cdots,p_r)\in (\mathbb{R}_+)^r.$ In the special
case where all the $p_k=1,$ the generalized complex ellipsoid
$\Sigma(\mathbf{n};\mathbf{p})$ reduces to the unit ball in
$\mathbb{C}^{n_1+\cdots+n_r}.$ Also, it is known that a generalized
complex ellipsoid $\Sigma(\mathbf{n};\mathbf{p})$ is homogeneous if
and only if $p_k=1$ for all $1\leq k\leq r$ (cf. Kodama \cite{K1}).
In general, a generalized complex ellipsoid is not strongly
pseudoconvex and its boundary is not smooth. The automorphism group ${\rm Aut}(\Sigma(\mathbf{n};\mathbf{p}))$ of
$\Sigma(\mathbf{n};\mathbf{p})$ has been studied by Dini-Primicerio
\cite{Dini}, Kodama \cite{K1} and Kodama-Krantz-Ma
\cite{Kodama-Krantz-Ma}.

In 2013, Kodama \cite{K1} obtained the result as follows.

\begin{Theorem} [Kodama \cite{K1}]\label{Thm K} {\it $(i)$
If $1$ does not appear in $p_1,\cdots,p_r$, then any automorphism
$\varphi\in {\rm Aut}(\Sigma(\mathbf{n};\mathbf{p}))$ is of the form
\begin{equation}\label{eq1.1}
\varphi(\zeta_1,\cdots,\zeta_r)=(\gamma_1(\zeta_{\sigma(1)}),\cdots,\gamma_r(\zeta_{\sigma(r)}))
\end{equation}
where $\sigma\in S_r$ is a permutation of the $r$ numbers
$\{1,\cdots,r\}$ such that $n_{\sigma(i)}=n_i, p_{\sigma(i)}=p_i$,
$1\leq i\leq r$ and $\gamma_1,\cdots, \gamma_r$ are unitary
transformation of $\mathbb{C}^{n_1}(n_{\sigma(1)}=n_1),\cdots,
\mathbb{C}^{n_r}(n_{\sigma(r)}=n_r)$ respectively.

$(ii)$ If $1$ appears in $p_1,\cdots,p_r$, we can assume, without
loss of generality, that $p_1=1, p_2\neq 1,\cdots,p_r\neq 1$, then
${\rm Aut}(\Sigma(\mathbf{n};\mathbf{p}))$ is generated by elements
of the form \eqref{eq1.1} and automorphisms of the form
\begin{equation}
\varphi_a(\zeta_1,\zeta_2,\cdots,\zeta_r)=(T_{a}(\zeta_1),\zeta_2(\psi_a(\zeta_1))^{1/2p_2},\cdots,\zeta_r(\psi_a(\zeta_1))^{1/2p_r})
\end{equation}
where $T_a$ is an automorphism of the ball $\mathbb{B}^{n_1}$ in
$\mathbb{C}^{n_1}$, which sends a point $a\in \mathbb{B}^{n_1}$ to
the origin and
$$\psi_a(\zeta_1)=\frac{1-\|a\|^2}{(1-\langle
\zeta_1,a\rangle)^2}.$$}
\end{Theorem}

In this paper, we define the generalized Fock-Bargmann-Hartogs
domains $D_{n_{0}}^{\mathbf{n},\mathbf{p}}(\mu)$ as follows:
$$D_{n_{0}}^{\mathbf{n},\mathbf{p}}(\mu)=\{(z,w_{(1)},\cdots,w_{(l)})\in \mathbb{C}^{n_{0}}\times \mathbb{C}^{n_{1}}\times \cdots\times\mathbb{C}^{n_{l}}
  :\sum\limits_{j=1}^{l}{\left\lVert w_{(j)}\right\rVert}^{2p_{j}}<e^{-\mu{\left\lVert z\right\rVert}^{2}}\}\quad (\mu>0),$$
where $\mathbf{p}=(p_{1},\cdots,p_{l})\in(\mathbb{R}_{+})^{l},\;
\mathbf{n}=(n_{1},\cdots,n_{l}), \;
w_{(j)}=(w_{j1},\cdots,w_{jn_{j}})\in\mathbb{C}^{n_{j}},$ in which
$n_{j}$ is a positive integer for $1\leq j\leq l.$ Here and
henceforth, with no loss of generality, we always assume that $p_{i}\neq 1$
($2\leq i\leq l$) for $D_{n_{0}}^{\mathbf{n},\mathbf{p}}(\mu)$.

Obviously, each generalized
Fock-Bargmann-Hartogs domain $D_{n_{0}}^{\mathbf{n},\mathbf{p}}$ is an
unbounded non-hyperbolic domain. In general, a generalized
Fock-Bargmann-Hartogs domain is not a strongly pseudoconvex domain and its
boundary is not smooth.

In this paper, we prove the following results.

\begin{Theorem}\label{Thm1}Suppose $D_{n_{0}}^{\mathbf{n},\mathbf{p}}(\mu)$ and $D_{m_{0}}^{\mathbf{m},\mathbf{q}}(\nu)$ are two equidimensional generalized Fock-Bargmann-Hartogs domains.
Let $f:\; D_{n_{0}}^{\mathbf{n},\mathbf{p}}(\mu)\rightarrow D_{m_{0}}^{\mathbf{m},\mathbf{q}}(\nu)$
be a biholomorphic mapping. Then there exists $\phi\in\mathrm{Aut}(D_{m_{0}}^{\mathbf{m},\mathbf{q}}(\nu))$ such that
\begin{equation}\label{eq3}
\phi\circ f(z,w)=(z,w_{(\sigma(1))},\cdots,w_{(\sigma(l))})
\left(\begin{array}{ccccc}
A&\null&\null&\null&\null\\
\null&\Gamma_{1}&\null&\null&\null\\
\null&\null&\Gamma_{2}&\null&\null\\
\null&\null&\null&\ddots&\null\\
\null&\null&\null&\null&\Gamma_{l}\\
\end{array}\right),
\end{equation}
where
$\sigma \in S_{l}$ is a permutation such that $n_{\sigma(j)}=m_{j}$, $p_{\sigma(j)}=q_{j}\; (1\leq j\leq l)$,
 $\sqrt{\frac{\nu}{\mu}}A\in \mathcal{U}(n)$  $(n:=n_{0}=m_{0})$, and $\Gamma_{i}\in \mathcal{U}(m_{i})\; (1\leq i\leq l)$.
\end{Theorem}

\begin{Corollary}\label{Coro.1.5}
Let $f:\; D_{n_{0}}^{\mathbf{n},\mathbf{p}}(\mu)\rightarrow D_{n_{0}}^{\mathbf{n},\mathbf{p}}(\mu)$ be a biholomorphic mapping with $f(0)=0$. Then we have
\begin{equation*}
f(z,w)=(z,w_{(\sigma(1))},\cdots,w_{(\sigma(l))})
\left(\begin{array}{ccccc}
A&\null&\null&\null&\null\\
\null&\Gamma_{1}&\null&\null&\null\\
\null&\null&\Gamma_{2}&\null&\null\\
\null&\null&\null&\ddots&\null\\
\null&\null&\null&\null&\Gamma_{l}\\
\end{array}\right),
\end{equation*}
where  $\sigma \in S_{l}$ is a permutation such that $n_{\sigma(j)}=n_{j}$, $p_{\sigma(j)}=p_{j} \;(1\leq j\leq l)$,
 $A\in \mathcal{U}(n_0)$ and $\Gamma_{i}\in \mathcal{U}(n_{i}) \; (1\leq i\leq l)$.
\end{Corollary}

As a consequence, it is easy for us to prove the following results.

\begin{Theorem}\label{Thm.1.6}
The automorphism
group $\mathrm{Aut}(D_{n_{0}}^{\mathbf{n},\mathbf{p}}(\mu))$ is generated by the following
mappings:
$$
\varphi_{A}:(z,w_{(1)},\cdots,w_{(l)})\longmapsto
(zA,w_{(1)},\cdots,w_{(l)});$$
$$\varphi_{D}:(z,w_{(1)},\cdots,w_{(l)})\longmapsto (z,(w_{(\sigma(1))},\cdots,w_{(\sigma(l))})D);$$
$$\varphi_{a}:(z,w)\longmapsto(z+a,w_{(1)}(e^{-2\mu{<z,a>}-\mu{\left\lVert a\right\rVert}^{2}})^{\frac{1}{2p_{1}}},\cdots,w_{(l)}(e^{-2\mu{<z,a>}-\mu{\left\lVert a\right\rVert}^{2}})^{\frac{1}{2p_{l}}}),$$
where $a\in \mathbb{C}^{n_0}$,  $A\in \mathcal{U}(n_{0})$,  $\sigma \in S_{l}$ is a permutation such that $n_{\sigma(j)}=n_{j}$, $p_{\sigma(j)}=p_{j}\; (1\leq j\leq l)$, and
\begin{equation*}
D=
\left(\begin{array}{cccc}
\Gamma_{1}&\null&\null&\null\\
\null&\Gamma_{2}&\null&\null\\
\null&\null&\ddots&\null\\
\null&\null&\null&\Gamma_{l}\\
\end{array}\right),
\end{equation*}
in which $\Gamma_{i}\in \mathcal{U}(n_{i})$ ($1\leq i \leq l$).
\end{Theorem}

Now, for $\mathbf{p}$ and $\mathbf{q}$,  we introduce notation:
\begin{equation*}
\epsilon=\begin{cases}
1,\;&p_{1}=1\\
0,\;&p_{1}\neq1\\
\end{cases},\;\;\;
\delta=\begin{cases}
1,\;&q_{1}=1\\
0,\;&q_{1}\neq1\\
\end{cases}.
\end{equation*}

\begin{Theorem}\label{Thm1.2}Suppose $D_{n_{0}}^{\mathbf{n},\mathbf{p}}(\mu)$ and $D_{m_{0}}^{\mathbf{m},\mathbf{q}}(\nu)$ are two equidimensional generalized Fock-Bargmann-Hartogs domains with $\min \{n_{1+\epsilon}, n_{2},\cdots,n_{l},n_{1}+\cdots+n_{l}\}\geq 2$ and $\min \{m_{1+\delta}, m_{2}, \cdots,m_{l},m_{1}+\cdots+m_{l}\}\geq 2$. Then any proper holomorphic mapping  between $D_{n_{0}}^{\mathbf{n},\mathbf{p}}(\mu)$ and $D_{m_{0}}^{\mathbf{m},\mathbf{q}}(\nu)$ must be a biholomorphism.
\end{Theorem}
\begin{Remark}
The conditions $\min\{n_{1+\epsilon}, n_{2},\cdots,n_{l}\}\geq 2$ can not be removed. For example, $n_{1}=1$ (i.e, $w_{1}\in\mathbb{C}$), $p_{1}\neq1$, and
$$F(z,w):(z,w_{(1)},\cdots,w_{(l)})\rightarrow(z,w_{(1)}^{2},w_{(2)},\cdots,w_{(l)}).$$
Then $F$ is a proper holomorphic mapping between $D_{n_{0}}^{\mathbf{n},\mathbf{p}}(\mu)$ and $D_{n_{0}}^{\mathbf{n},\mathbf{q}}(\mu)$ where $\mathbf{q}=(p_{1}/2,p_{2}, \cdots,p_{l})$. $F$ is not a biholomorphism.
\end{Remark}

\begin{Corollary}Suppose $D_{n_{0}}^{\mathbf{n},\mathbf{p}}(\mu)$ is a generalized Fock-Bargmann-Hartogs domain with $$\min \{n_{1+\epsilon}, n_{2},\cdots,n_{l},n_{1}+\cdots+n_{l}\}\geq 2.$$ Then any proper holomorphic self-mapping of $D_{n_{0}}^{\mathbf{n},\mathbf{p}}(\mu)$ must be an automorphism.
\end{Corollary}

\begin{Remark}
The conditions $n_{1}+\cdots+n_{l}\geq 2$ can not be removed. For instance, with no loss of generality, we can assume $n_{1}=1$ and $n_{i}=0$ ($2\leq i\leq l$). Then
$$F:(z,w_{(1)})\rightarrow (\sqrt{2}z,w_{(1)}^{2})$$
is a  proper holomorphic self-mapping of $D_{n_{0}}^{\mathbf{n},\mathbf{p}}(\mu)$ which is not an automorphism.
\end{Remark}

The paper is organized as follows. In Section 2, using the explicit formula for the Bergman kernels of the generalized Fock-Bargmann-Hartogs domains, we prove
 that a proper holomorphic mapping between two equidimensional generalized Fock-Bargmann-Hartogs domains extends holomorphically to their closures and check that the Cartan's theorem holds also for the generalized
Fock-Bargmann-Hartogs domains. In Section 3, we  exploit the boundary structure of  generalized Fock-Bargmann-Hartogs domains to prove our results in this paper.

\section{Preliminaries}
\subsection{The Bergman kernel of the domain $D_{n_{0}}^{\mathbf{n},\mathbf{p}}$}

For a domain $\Omega$ in $\mathbb{C}^n$,  let $A^2(\Omega)$ be the
Hilbert space of square integrable holomorphic functions on $\Omega$
with the inner product:
$$\langle f,g\rangle=\int_{\Omega}f(z)\overline{g(z)} dV(z)\;\;(f,g\in \mathcal{O}(\Omega)),$$
where $dV$ is the Euclidean volume form. The Bergman kernel $K(z,w)$
of $A^2(\Omega)$ is defined as the reproducing kernel of the Hilbert
space $A^2(\Omega)$, that is, for all $f\in A^2(\Omega),$ we have
$$f(z)=\int_{\Omega}f(w)K(z,w)dV(w) \;\;(z\in\Omega).$$
For a positive continuous function $p$ on $\Omega$, let
$A^2(\Omega,p)$ be the weighted Hilbert space of square integrable
holomorphic functions with respect to the weight function $p$ with
the inner product:
$$\langle f,g\rangle=\int_{\Omega}f(z)\overline{g(z)}p(z)dV(z)\;\;\; (f,g\in \mathcal{O}(\Omega)).$$
Similarly, the weighted Bergman kernel $K_{A^2(\Omega,p)}$ of
$A^2(\Omega,p)$ is defined as the reproducing kernel of the Hilbert
space $A^2(\Omega,p)$. For a positive integer $m$, define the
Hartogs domain $\Omega_{m,p}$ over $\Omega$  by
$$\Omega_{m,p}=\{(z,w)\in\Omega\times\mathbb{C}^m: \|w\|^2<p(z)\}.$$

Ligocka \cite{L1, L2} showed that the Bergman kernel of
$\Omega_{m,p}$ can be expressed as infinite sum in terms of the
weighted Bergman kernel of $A^2(\Omega,p^k)\;(k=1,2,\cdots)$ as
follows.

\begin{Theorem}[Ligocka \cite{L2}]\label{Thm.2.1} {\it Let
$K_m$ be the Bergman kernel of $\Omega_{m,p}$ and let
$K_{A^2(\Omega,p^k)}$ be the weighted Bergman kernel of
$A^2(\Omega,p^k)$ $(k=1,2,\cdots)$. Then
$$K_m((z,w),(t,s))=\frac{m!}{\pi^m}\sum_{k=0}^{\infty}\frac{(m+1)_k}{k!}K_{A^2(\Omega,p^{k+m})}(z,t) \langle w,s\rangle^k,$$
where $(a)_k$ denotes the Pochhammer symbol
$(a)_k=a(a+1)\cdots(a+k-1)$. }
\end{Theorem}

\vskip 6pt The Fock-Bargmann space is the weighted Hilbert space
$A^2(\mathbb{C}^n,e^{-\mu\|z\|^2})$ on $\mathbb{C}^n$ with the
Gaussian weight function $e^{-\mu\|z\|^{2}}\;(\mu>0)$.  The
reproducing kernel of $A^2(\mathbb{C}^n,e^{-\mu\|z\|^2})$, called
the Fock-Bargmann kernel, is ${\mu^n e^{\mu \langle
z,t\rangle}}/{\pi^n}$ (see Bargmann \cite{Ba}). Thus, the
Fock-Bargmann-Hartogs domain $D_{n,m}=\{(z,w)\in
\mathbb{C}^{n}\times \mathbb{C}^{m}: \left\lVert w \right\rVert^{2}
< e^{-\mu{\left\lVert z\right\rVert}^{2}}\}\;(\mu>0)$ and the
Fock-Bargmann space $A^2(\mathbb{C}^n,e^{-\mu\|z\|^2})$ are closely
related. In 2013, using Theorem \ref{Thm.2.1} and the expression of the
Fock-Bargmann kernel, Yamamori \cite{Y} gave the Bergman kernel of
the Fock-Bargmann-Hartogs domain $D_{n,m}$ as follows.

\vskip 6pt

\begin{Theorem} [Yamamori \cite{Y}] {\it The Bergman
kernel of the Fock-Bargmann-Hartogs domain $D_{n,m}$ is given
by
\begin{align*}
 \hskip 12pt K_{D_{n,m}}((z,w),(t,s))
=\frac{m!\mu^n}{\pi^{m+n}}\sum_{k=0}^{\infty}\frac{(m+1)_k(k+m)^n}{k!}
e^{\mu(k+m)\langle z,t\rangle} \langle w,s\rangle^k,
\end{align*}
where $(a)_k$ denotes the Pochhammer symbol
$(a)_k=a(a+1)\cdots(a+k-1)$.}
\end{Theorem}

Following the idea of Theorem \ref{Thm.2.1}, we compute the Bergman kernel for the generalized
Fock-Bargmann-Hartogs domain $D_{n_{0}}^{\mathbf{n},\mathbf{p}}$. In order
to compute the Bergman kernel, we first introduce some notation.

Let
$$\alpha=(\alpha_{(1)},\cdots,\alpha_{(l)})\in
(\mathbb{R}_{+})^{n_{1}}\times
\cdots\times(\mathbb{R}_{+})^{n_{l}},$$ where
$\alpha_{(i)}=(\alpha_{i1},\cdots,\alpha_{in_{i}})\in
(\mathbb{R}_{+})^{n_{i}}$ for $1\leq i\leq l.$
For $\alpha\in(\mathbb{R}_{+})^{n}$, we define 
$$\beta(\alpha)=\frac{\prod_{i=1}^l\Gamma(\alpha_{i})}{\Gamma(\vert \alpha
\vert)};$$ 
see D'Angelo \cite{D2}. Here $\Gamma$ is the usual Euler Gamma function.

\begin{Lemma}[D'Angelo \cite{D2}, Lemma 1]\label{Thm.D}
Suppose $\alpha \in (\mathbb{R}_{+})^{n}$. Then we have
$$\int_{B_{+}^{n}}{r^{2\alpha-1}d\textrm{V}(r)}=\frac{\beta(\alpha)}{2^{n}\vert \alpha \vert},$$
$$\int_{{S}_{+}^{n-1}}w^{2\alpha-1}d\sigma(w)=\frac{\beta(\alpha)}{2^{n-1}},$$
where $d\textrm{V}$ is the Euclidean $n$-dimensional volume form,
$d\textrm{S}$ is the Euclidean $(n-1)$-dimensional volume form, and
the subscript $``+"$ denotes that all the variables are positive,
that is,  $B_{+}^{n}= B^{n}\cap (\mathbb{R}_{+})^{n}$ and
$S_{+}^{n-1}= S^{n-1}\cap (\mathbb{R}_{+})^{n}$, in which $B^{n}$ is
the unit ball in $\mathbb{R}^{n}$ and $S^{n-1}$ is the unit sphere
in $\mathbb{R}^{n}$.
\end{Lemma}
\vskip 5pt

\begin{Theorem}\label{Thm2.4}
Suppose
$\alpha=(\alpha_{(1)},\cdots,\alpha_{(l)})\in
(\mathbb{R}_{+})^{n_{1}}\times
\cdots(\mathbb{R}_{+})^{n_{l}},\;\alpha_{(i)}=(\alpha_{i1},\cdots,\alpha_{in_{i}})\in
(\mathbb{R}_{+})^{n_{i}}, 1\leq i\leq l.$ Then we have the
formula:
\begin{equation}\label{eq2.4}
\begin{aligned}
\int_{\sum\limits_{j=1}^{l}{\left\lVert w_{(j)}\right\rVert}^{2p_{j}}<t}w^{\alpha}\bar{w}^{\alpha}d\textrm{V}(w)
=(\pi)^{n_{1}+\cdots+n_{l}}\frac{\prod\limits_{i=1}^{l}\Gamma(\alpha_{i}+1)
\prod\limits_{i=1}^{l}\Gamma(\frac{\vert\alpha_{(i)}\vert+n_{i}}{p_{i}})}
{\prod\limits_{i=1}^{l}p_{i}\prod\limits_{i=1}^{l}\Gamma(\vert \alpha_{i}\vert+n_{i})\Gamma (\sum\limits_{i=1}^{l}\frac{\vert\alpha_{(i)}\vert+n_{i}}{p_{i}}+1)}\cdot t^{\sum\limits_{i=1}^{l}\frac{\vert\alpha_{(i)}
\vert+n_{i}}{p_{i}}}\\
\end{aligned}
\end{equation}
\end{Theorem}

\begin{proof}[Proof]
For the integral
\begin{equation}\label{eq2.5}
\int_{\sum\limits_{j=1}^{l}{\left\lVert w_{(j)}\right\rVert}^{2p_{j}}<t}w^{\alpha}\bar{w}^{\alpha}d\textrm{V}(w),
\end{equation}
by applying the polar coordinates $w=se^{i\theta}$ (namely,
$w_{ij}=s_{ij}e^{i\theta_{ij}},$ $1\leq j\leq n_{i}, \; 1\leq i\leq
l$, $s=(s_{(1)},\cdots,s_{(l)})$), we have
\begin{equation*}
\begin{aligned}
\eqref{eq2.5}=(2\pi)^{n_{1}+\cdots+n_{l}}\int_{\sum\limits_{j=1}^{l}{\left\|s_{(j)}\right\|}^{2p_{j}}<t\atop s_{ji}>0,1\leq i\leq n_{j}, 1\leq j\leq l}s^{2\alpha+1}d\textrm{V}(s).
\end{aligned}
\end{equation*}
Using the spherical coordinates in the variables $s_{(1)},s_{(2)},\cdots,s_{(l)}$ respectively, we get
\begin{equation*}
\begin{aligned}
&\int_{{\sum\limits_{j=1}^{l}{\left\lVert s_{(j)}\right\rVert}^{2p_{j}}<t\atop s_{ji}>0,\;1\leq i\leq n_{j},\;1\leq j\leq l}}
s^{2\alpha+1}d\textrm{V}(s)\\
=&\int_{\sum\limits_{i=1}^{l}{\vert\rho_{i}\vert}^{2p_{i}}<t\atop
\rho_{i}>0,1\leq i\leq l}\rho_{1}^{2\vert\alpha_{(1)}\vert+2n_{1}-1}
\cdots\rho_{l}^{2\vert\alpha_{(l)}\vert+2n_{l}-1}d\rho_{1}d\rho_{2}\cdots d\rho_{l}\\
&\times\int_{S_{+}^{n_{1}-1}}\cdots\int_{S_{+}^{n_{l}-1}}w_{(1)}^{2\alpha_{(1)}+1}\cdots
w_{(l)}^{2\alpha_{(l)}+1}
d\sigma(w_{(1)})\cdots d\sigma(w_{(l)}).\\
\end{aligned}
\end{equation*}
Let ${\rho_{i}}^{p_{i}}=r_{i},1\leq i\leq l$. Then we have
$d\rho_{i}=\frac{1}{p_{i}}\rho_{i}^{1-p_{i}}dr_{i}
=\frac{1}{p_{i}}r_{i}^{\frac{1}{p_{i}}-1}dr_{i}$. Therefore,  Lemma \ref{Thm.D} and  the
above formulas yield
\begin{equation*}
\begin{aligned}
\eqref{eq2.5}
=(2\pi)^{n_{1}+\cdots+n_{l}}\frac{1}{\prod\limits_{i=1}^{l}p_{i}}\frac{
\beta(\alpha_{(1)}+1)}{2^{n_{1}-1}}
\cdots\frac{\beta(\alpha_{(l)}+1)}{2^{n_{l}-1}}\int_{\sum\limits_{i=1}^{l}{\vert r_{i} \vert}^{2}<t \atop
r_{i}>0,1\leq i\leq l}
r_{1}^{\frac{2\vert\alpha_{(1)}\vert+2n_{1}}{p_{1}}-1}
\cdots r_{l}^{\frac{2\vert\alpha_{(l)}\vert+2n_{l}}{p_{l}}-1}dr_{1}\cdots dr_{l}.\\
\end{aligned}
\end{equation*}
Let $r=(r_{1},r_{2},\cdots,r_{l})\in (\mathbb{R}_{+})^{l}$ and
$k:=t^{-\frac{1}{2}}r$. Then $dr=t^{\frac{l}{2}}dk$. After a straightforward computation, we obtain
that \eqref{eq2.5} equals
\begin{equation*}
\begin{aligned}
(2\pi)^{n_{1}+\cdots+n_{l}}\frac{1}{\prod\limits_{i=1}^{l}p_{i}}\frac{
\beta(\alpha_{(1)}+1)}{2^{n_{1}-1}}
\cdots\frac{\beta(\alpha_{(l)}+1)}{2^{n_{l}-1}}\cdot t^{\sum\limits_{i=1}^{l}\frac{\vert\alpha_{(i)}\vert+n_{i}}{p_{i}}}
\int_{B_{+}^{l}}
k_{1}^{\frac{2\vert\alpha_{(1)}\vert+2n_{1}}{p_{1}}-1}
\cdots k_{l}^{\frac{2\vert\alpha_{(l)}\vert+2n_{l}}{p_{l}}-1}dk_{1}\cdots dk_{l}.
\end{aligned}
\end{equation*}
Applying Lemma \ref{Thm.D} to the above formula, we get
\begin{equation}
\begin{aligned}
\eqref{eq2.5}
=&(\pi)^{n_{1}+\cdots+n_{l}}\beta(\alpha_{(1)}+1)\cdots
\beta(\alpha_{(l)}+1)\frac{\beta(\alpha^{'})}{\vert\alpha^{'}\vert\prod\limits_{i=1}^{l}p_{i}}
\cdot t^{\sum\limits_{i=1}^{l}\frac{\vert\alpha_{(i)}
\vert+n_{i}}{p_{i}}}\\
=&(\pi)^{n_{1}+\cdots+n_{l}}\frac{1}{\prod\limits_{i=1}^{l}p_{i}}\frac{\prod\limits_{i=1}^{l}\Gamma(\alpha_{(i)}+1)
\prod\limits_{i=1}^{l}\Gamma(\frac{\vert\alpha_{(i)}\vert+n_{i}}{p_{i}})}
{\prod\limits_{i=1}^{l}\Gamma(\vert \alpha_{i}\vert+n_{i})\Gamma (\sum\limits_{i=1}^{l}\frac{\vert\alpha_{(i)}\vert+n_{i}}{p_{i}}+1)}\cdot t^{\sum\limits_{i=1}^{l}\frac{\vert\alpha_{(i)}
\vert+n_{i}}{p_{i}}},\\
\end{aligned}
\end{equation}
where
$\alpha^{'}=(\frac{\vert\alpha_{(1)}\vert+n_{1}}{p_{1}},\cdots,\frac{\vert\alpha_{(l)}\vert+n_{l}}{p_{l}})
\in (\mathbb{R}_{+})^{l}$.
\end{proof}

Now we consider
 the Hilbert space $A^{2}(D_{n_{0}}^{\mathbf{n},\mathbf{p}}(\mu))$ of square-integrable holomorphic functions on $D_{n_{0}}^{\mathbf{n},\mathbf{p}}(\mu)$.

\begin{Lemma}\label{lemma2.5}Let $f\in A^{2}(D_{n_{0}}^{\mathbf{n},\mathbf{p}}(\mu))$. Then \\
$$f(z,w)=\sum_{\alpha} f_{\alpha}(z)w^{\alpha},$$
where the  series is uniformly convergent on compact subsets of
$D_{n_{0}}^{\mathbf{n},\mathbf{p}}(\mu),$  $f_{\alpha}(z)\in
A^{2}(\mathbb{C}^{n_{0}},e^{-\mu\lambda_{\alpha}{\left\lVert
z\right\rVert}^{2}})$ for any  $\alpha=(\alpha_{(1)},\cdots
,\alpha_{(l)})\in\mathbb{N}^{n_{1}}\times \cdots\times
\mathbb{N}^{n_{l}},\;\alpha_{(i)}=(\alpha_{i1},\cdots,
\alpha_{in_{i}})\in \mathbb{N}^{n_{i}},1\leq i\leq l$,
$\lambda_{\alpha}=\sum\limits_{i=1}^{l}\frac{\vert\alpha_{(i)}\vert+n_{i}}{p_{i}}$,
in which $A^{2}(\mathbb{C}^{n},e^{-\mu\lambda_{\alpha}{\left\lVert
z\right\rVert}^{2} })$ denotes the space of square-integrable
holomorphic functions on $\mathbb{C}^{n}$ with respect to
the measure $e^{-\mu\lambda_{\alpha}{\left\lVert z\right\rVert}^{2}}\textrm{d}V_{2n}$.\\
\end{Lemma}

\begin{proof}[Proof]
Since $D_{n_{0}}^{\mathbf{n},\mathbf{p}}(\mu)$ is a
complete Reinhardt domain, each holomorphic function on
$D_{n_{0}}^{\mathbf{n},\mathbf{p}}(\mu)$ is the sum of a locally uniformly
convergent power series. Thus, for $f\in
A^{2}(D_{n_{0}}^{\mathbf{n},\mathbf{p}}(\mu)),$ we have
$$f(z,w)=\sum_{\alpha} f_{\alpha}(z)w^{\alpha},$$
where the  series is uniformly convergent on compact subsets of
$D_{n_{0}}^{\mathbf{n},\mathbf{p}}(\mu).$  We choose a sequence of
compact subsets $D_{k}\; (1\leq k < \infty)$
 $$D_{k}:=\{(z,w_{(1)},\cdots,w_{(l)})\in
\mathbb{C}^{n_{0}}\times \mathbb{C}^{n_{1}}\times
\cdots\times\mathbb{C}^{n_{l}}
  :\sum\limits_{j=1}^{l}{\left\lVert w_{(j)}\right\rVert}^{2p_{j}}
\leq e^{-\mu{\left\lVert z\right\rVert}^{2}}-\frac{1}{k}\}\cap
\overline{B(0,k)},$$
where ${B(0,k)}$ is the ball in $\mathbb{C}^{n_{0}+n_{1}+\cdots+n_{l}}$
of the radius $k$. Obviously, $D_{k}\Subset D_{k+1}$ and
$\bigcup\limits_{k=1}^{\infty}D_{k}=D_{n_{0}}^{\mathbf{n},\mathbf{p}}(\mu)$.
Since $D_{k}$ is a circular domain, then
$$f_{\alpha}(z)w^{\alpha}\perp f_{\beta}(z)w^{\beta}\qquad
(\alpha\neq\beta)$$ in the Hilbert space $A^{2}(D_{k})$. Hence we
have $${\left\lVert
f\right\rVert}_{L^{2}(D_{k})}^{2}=\sum\limits_{\vert\alpha\vert=0}^{\infty}{\left\lVert
f_{\alpha}(z)w^{\alpha}\right\rVert}_{L^{2}(D_{k})}^{2}.$$ Since
$f(z,w)\in A^{2}(D_{n_{0}}^{\mathbf{n},\mathbf{p}}(\mu))$, we have
$${\left\lVert
f_{\alpha}(z)w^{\alpha}\right\rVert}_{L^{2}(D_{k})}^{2}\leq
{\left\lVert f\right\rVert}_{L^{2}(D_{k})}^{2}\leq {\left\lVert
f\right\rVert}_{L^{2}(D_{n_{0}}^{\mathbf{n},\mathbf{p}}(\mu))}^{2}.$$
Then $f_{\alpha}(z)w^{\alpha}\in A^{2}(D_{n_{0}}^{\mathbf{n},\mathbf{p}}(\mu))$.
Therefore,
$$\int_{D_{n_{0}}^{\mathbf{n},\mathbf{p}}(\mu)}{\vert
f_{\alpha}(z)\vert}^{2}w^{\alpha}\bar{w}^{\alpha}\textrm{d}V<\infty$$
$$\Longrightarrow \int_{\mathbb{C}^{n_{0}}}{\vert f_{\alpha}(z)\vert}^{2}\textrm{d}V(z)\int_{\sum\limits_{j=1}^{l}{\Vert w_{(j)}\Vert}^{2p_{j}}<e^{-\mu{\left\lVert z\right\rVert}^{2}}}w^{\alpha}\bar{w}^{\alpha}\textrm{d}V(w)<\infty.$$
By \eqref{eq2.4}, it follows
 $$\int_{\mathbb{C}^{{n_{0}}}}{\vert f_{\alpha}(z)\vert}^{2}e^{-\mu\lambda_{\alpha}{\left\lVert z\right\rVert}^{2}}\textrm{d}V(z)<\infty.$$
Consequently, $f_{\alpha}(z)\in A^{2}(\mathbb{C}^{n_{0}},e^{-\mu\lambda_{\alpha}{\left\lVert z\right\rVert}^{2}}),$ where $\lambda_{\alpha}=\sum\limits_{i=1}^{l}\frac{\vert\alpha_{(i)}\vert+n_{i}}{p_{i}}$.
\end{proof}

 Lemma \ref{lemma2.5} implies that $f(z)w^{\alpha}$ where $f(z)\in A^{2}(\mathbb{C}^{n_{0}},e^{-\mu\lambda_{\alpha}{\left\lVert z\right\rVert}^{2}})$ form a linearly dense subset of
$A^{2}(D_{n_{0}}^{\mathbf{n},\mathbf{p}}(\mu))$.  Now we can express the Bergman kernel of $D_{n_{0}}^{\mathbf{n},\mathbf{p}}(\mu)$ as follows.

\vskip 5pt

\begin{Theorem}\label{4}The Bergman
kernel of $D_{n_{0}}^{\mathbf{n},\mathbf{p}}(\mu)$ can be expressed by the
following form
\begin{equation}\label{eq2.7}
K_{D_{n_{0}}^{\mathbf{n},\mathbf{p}}(\mu)}[(z,w),(s,t)]=
\sum\limits_{\vert\alpha\vert=0}^{\infty}c_{\alpha}
\frac{\lambda_{\alpha}^{n_{0}}\mu^{n_{0}}}{\pi^{n_{0}}}e^{\lambda_{\alpha}\mu\langle
z,s\rangle}w^{\alpha}\bar{t}^{\alpha},
\end{equation}
where $\alpha=(\alpha_{(1)},\cdots
,\alpha_{(l)})\in\mathbb{N}^{n_{1}}\times \cdots\times
\mathbb{N}^{n_{l}},\;\alpha_{(i)}=(\alpha_{i1},\cdots,
\alpha_{in_{i}})\in \mathbb{N}^{n_{i}},1\leq i\leq l$, and
$$c_{\alpha}=\frac{\prod\limits_{i=1}^{l}p_{i}\prod\limits_{i=1}^{l}\Gamma(\vert \alpha_{i}\vert+n_{i})
\Gamma
(\sum\limits_{i=1}^{l}\frac{\vert\alpha_{(i)}\vert+n_{i}}{p_{i}}+1)}{(\pi)^{n_{1}+\cdots+n_{l}}\prod\limits_{i=1}^{l}\Gamma(\alpha_{(i)}+1)
\prod\limits_{i=1}^{l}\Gamma(\frac{\vert\alpha_{(i)}\vert+n_{i}}{p_{i}})},\;\;\lambda_{\alpha}=\sum\limits_{i=1}^{l}
\frac{\vert\alpha_{(i)}\vert+n_{i}}{p_{i}}.$$
\end{Theorem}

\begin{proof}[Proof] Since $D_{n_{0}}^{\mathbf{n},\mathbf{p}}(\mu)$ is a complete Reinhardt domain,
it follows
$$K_{D_{n_{0}}^{\mathbf{n},\mathbf{p}}(\mu)}[(z,w),(s,t)]=\sum\limits_{\vert\beta\vert=0}^{\infty}
c_{\beta}g_{\beta}(z,s)w^{\beta}\bar{t}^{\beta},$$
where the sum is locally uniformly convergent,
by the invariance of the Bergman
kernel $K_{D_{n_{0}}^{\mathbf{n},\mathbf{p}}(\mu)}$ on $D_{n_{0}}^{\mathbf{n},\mathbf{p}}(\mu)$ under the unitary subgroup action $$(z_1,\cdots, z_{n_0+|\mathbf{n}|}) \rightarrow
(e^{\sqrt{-1}\theta_1}z_1, \cdots, e^{\sqrt{-1}\theta_{n_0+|\mathbf{n}|}}z_{n_0+|\mathbf{n}|}) \; \;(\theta_1,\cdots, \theta_{n_0+|\mathbf{n}|}\in\mathbb{R}).$$

For any
$\alpha=(\alpha_{(1)},\cdots,\alpha_{(l)})\in
\mathbb{N}^{n_{1}}\times \cdots \times
\mathbb{N}^{n_{l}}$ with $\alpha_{(i)}=(\alpha_{i1},\cdots,\alpha_{in_{i}})\in
\mathbb{N}^{n_{i}}\;(1\leq i\leq l)$,  any $f(z)\in
A^{2}(\mathbb{C}^{n_{0}},e^{-\mu\lambda_{\alpha
}{\left\lVert z\right\rVert}^{2}})\; (\lambda_{\alpha}=\sum\limits_{i=1}^{l}\frac{\vert\alpha_{(i)}\vert+n_{i}}{p_{i}})$, we have $f(z)w^{\alpha}\in A^{2} (D_{n_{0}}^{\mathbf{n},\mathbf{p}}(\mu)) $.
Thus
\begin{equation*}
\begin{aligned}
f(z)w^{\alpha}&=\int_{D_{n_{0}}^{\mathbf{n},\mathbf{p}}(\mu)}f(s)t^{\alpha}
K_{D_{n_{0}}^{\mathbf{n},\mathbf{p}}(\mu)}[(z,w),(s,t)]\textrm{d}V\\
&=\int_{\mathbb{C}^{n_{0}}}f(s)\sum\limits_{\beta=0}^{\infty}c_{\beta}g_{\beta}(z,s)w^{\beta}\textrm{d}V(s)
\int_{\sum\limits_{j=1}^{l}{\Vert t_{(j)}\Vert}^{2p_{j}}<e^{-\mu{\left\lVert s\right\rVert}^{2}}}t^{\alpha}\bar{t}^{\beta}\textrm{d}V(t)\\
&=w^{\alpha}\int_{\mathbb{C}^{n_{0}}}f(s)g_{\alpha}(z,s)[e^{-\mu{\left\lVert s\right\rVert}^{2}}]^{\sum\limits_{i=1}^{l}\frac{\vert\alpha_{(i)}\vert+n_{i}}{p_{i}}}d\textrm{V}(s)\;\;(\mathrm{by }\;\eqref{eq2.4}).
\end{aligned}
\end{equation*}
By Bargmann \cite{Ba}, we get that the Bergman kernel of
$A^{2}(\mathbb{C}^{n_{0}},e^{-\mu\lambda_{\alpha }{\left\lVert
z\right\rVert}^{2}})$ can be described by the form
\begin{equation}
K_{\alpha}(z,w)=\frac{\lambda_{\alpha}^{n_{0}}\mu^{n_{0}}}{\pi^{n_{0}}}e^{\lambda_{\alpha}\mu\langle z,w\rangle}.
\end{equation}
Thus we obtain
 $$g_{\alpha}(z,s)=\frac{\lambda_{\alpha}^{n_{0}}\mu^{n_{0}}}{\pi^{n_{0}}}e^{\lambda_{\alpha}\mu\langle z,s\rangle}.$$
This completes the proof.
\end{proof}

The transformation rule for Bergman kernels under proper
holomorphic mapping (e.g., Th. 1 in Bell \cite{Bell82}) is also
valid for unbounded domains (e.g., see Cor. 1 in Trybula
\cite{Trybula}). Note that the coordinate functions play a key role
in the approach of Bell \cite{Bell82} to extend proper holomorphic
mapping, but, in general, are no longer square integrable on
unbounded domains. In order to overcome the difficulty, by combining
the transformation rule for Bergman kernels under proper holomorphic
mapping in Bell \cite{Bell82} and our explicit form  \eqref{eq2.7} of the Bergman
kernel function for $D_{n_{0}}^{\mathbf{n},\mathbf{p}}(\mu)$,
we prove that a proper holomorphic mapping between two equidi-
mensional generalized Fock-Bargmann-Hartogs domains extends holomorphically to their closures as follows.

\begin{Lemma}\label{Thm2.3.2}

Suppose that $f:\; D_{n_{0}}^{\mathbf{n},\mathbf{p}}(\mu)\rightarrow D_{m_{0}}^{\mathbf{m},\mathbf{q}}(\nu)$ is a proper holomorphic mapping
between two equidimensional  generalized Fock-Bargmann-Hartogs
domains. Then  $f$ extends holomorphically to a
neighborhood of the closure $\overline{D_{n_{0}}^{\mathbf{n},\mathbf{p}}(\mu)}$.
\end{Lemma}

In fact,  using the explicit form  \eqref{eq2.7} of the Bergman
kernel function for $D_{n_{0}}^{\mathbf{n},\mathbf{p}}(\mu)$,  we immediately have Lemma \ref{Thm2.3.2} by a slightly modifying the proof of Th. 2.5 in Tu-Wang \cite{T.W}.

\subsection{Cartan's Theorem on the $D_{n_{0}}^{\mathbf{n},\mathbf{p}}$}

Suppose $D$ is a domain in $\mathbb{C}^{N}$ and let $K_{D}(z,w)$ be its Bergman kernel. From Ishi-Kai \cite{IsK}, we know that if the following conditions are satisfied:\\
$(a)\quad K_{D}(0,0)>0$;\\
$(b)\quad T_{D}(0,0)\textrm{ is positive definite},$\\
where $T_{D}$ is an $N\times N$ matrix
\begin{equation*}
T_{D}(z,w):=\left(\begin{array}{ccc}
\frac{\partial^{2}\log K_{D}(z,w)}{\partial z_{1}\partial \overline{w_{1}}}&\cdots&\frac{\partial^{2}\log K_{D}(z,w)}{\partial z_{1}\partial \overline{w_{N}}}\\
\vdots&\ddots\vdots\\
\frac{\partial^{2}\log K_{D}(z,w)}{\partial z_{N}\partial \overline{w_{1}}}&\cdots&\frac{\partial^{2}\log K_{D}(z,w)}{\partial z_{N}\partial \overline{w_{N}}}
\end{array}
\right).
\end{equation*}
Then the Cartan's theorem can also be applied to the case of unbounded circular domains. The above conditions are obviously satisfied by the bounded domain.

Kim-Ninh-Yamamori \cite{Kim} proved the following result.

\vskip 8pt

 \begin{Lemma}[Kim-Ninh-Yamamori \cite{Kim}, Th. 4]\label{Thm3} Suppose that $D$ is a circular domain and its Bergman kernel satisfies the above conditions $(a)$ and $(b)$. If $\varphi \;(\in \mathrm{Aut}(D))$  preserves the origin, then $\varphi$ is a linear mapping.
 \end{Lemma}

Ishi-Kai \cite{IsK} proved the generalization of Lemma \ref{Thm3} as follows.

\begin{Lemma} [Ishi-Kai \cite{IsK}, Prop. 2.1]\label{IsK}
 Let $D_{k}$ be a circular domain (not necessarily bounded)
in $\mathbb{C}^N$ with $0\in D_{k}\; (k=1,2),$  and let
$\varphi:D_1\rightarrow D_2 $ be a biholomorphism with
$\varphi(0)=0$. If $K_{D_k}(0,0)>0$ and $T_{D_k}(0,0)$ is positive
definite $(k=1,2)$, then $\varphi$ is linear.
\end{Lemma}

Therefore, by using the expressions of Bergman kernels of generalized Fock-Bargmann-Hartogs domains, we have the following result.

\begin{Theorem}\label{Thm2.2.1}Suppose that $\varphi:D_{n_{0}}^{\mathbf{n},\mathbf{p}}(\mu)\rightarrow D_{m_{0}}^{\mathbf{m},\mathbf{q}}(\nu)$ be a biholomorphic mapping
between two equidimensional generalized Fock-Bargmann-Hartogs domains with $\varphi(0)=0$. Then $\varphi$ is linear.
\end{Theorem}

\begin{proof}[Proof]
By using the expressions \eqref{eq2.7} of Bergman kernels of generalized Fock-Bargmann-Hartogs domains and a straightforward computation, we
show that the Bergman kernel of every generalized Fock-Bargmann-Hartogs domain satisfies the above conditions (a) and (b). So we get Th. \ref{Thm2.2.1} by Lemma \ref{IsK}.
\end{proof}

\section{Proof Of The Main Theorem}

To begin, we exploit the boundary structure of $D_{n_{0}}^{\mathbf{n},\mathbf{p}}(\mu)$ which is comprised of
$$bD_{n_{0}}^{\mathbf{n},\mathbf{p}}(\mu)=b_{0}D_{n_{0}}^{\mathbf{n},\mathbf{p}}(\mu)\cup  b_{1}D_{n_{0}}^{\mathbf{n},\mathbf{p}}(\mu)\cup b_{2}D_{n_{0}}^{\mathbf{n},\mathbf{p}}(\mu),$$
where
\begin{equation*}
\begin{aligned}
&b_{0}D_{n_{0}}^{\mathbf{n},\mathbf{p}}(\mu)\\
&:=\{(z,w_{(1)},\cdots,w_{(l)})\in \mathbb{C}^{n_{0}}\times \cdots\times\mathbb{C}^{n_{l}}
  :\sum\limits_{j=1}^{l}{\left\lVert w_{(j)}\right\rVert}^{2p_{j}}=e^{-\mu{\left\lVert z\right\rVert}^{2}}
  ,\;{\left\lVert w_{(j)}\right\rVert}^{2}\neq 0,\;    1+\epsilon \leq j\leq l\};\\
\end{aligned}
\end{equation*}
\begin{equation*}
\begin{aligned}
&b_{1}D_{n_{0}}^{\mathbf{n},\mathbf{p}}(\mu)\\
&:=\bigcup\limits_{j=1+\epsilon}^{l}\{(z,w_{(1)},\cdots,w_{(l)})\in \mathbb{C}^{n_{0}}\times \cdots\times\mathbb{C}^{n_{l}}
  :\sum\limits_{j=1}^{l}{\left\lVert w_{(j)}\right\rVert}^{2p_{j}}=e^{-\mu{\left\lVert z\right\rVert}^{2}}
  ,\;{\left\lVert w_{(j)}\right\rVert}^{2}=0,\;p_{j}> 1\};\\
\end{aligned}
\end{equation*}
\begin{equation*}
\begin{aligned}
&b_{2}D_{n_{0}}^{\mathbf{n},\mathbf{p}}(\mu)\\
&:=\bigcup\limits_{j=1+\epsilon}^{l}\{(z,w_{(1)},\cdots,w_{(l)})\in \mathbb{C}^{n_{0}}\times \cdots\times\mathbb{C}^{n_{l}}
  :\sum\limits_{j=1}^{l}{\left\lVert w_{(j)}\right\rVert}^{2p_{j}}=e^{-\mu{\left\lVert z\right\rVert}^{2}}
  ,\;{\left\lVert w_{(j)}\right\rVert}^{2}=0,\;p_{j}< 1\}.\\
\end{aligned}
\end{equation*}

Now we give the following proposition.

\begin{Proposition}\label{pro.2.4.1}$(1)$ The boundary $b_{0}D_{n_{0}}^{\mathbf{n},\mathbf{p}}(\mu)$ is a real analytic hypersurface in $\mathbb{C}^{n_{0}+n_{1}+\cdots+n_{l}}$ and $D_{n_{0}}^{\mathbf{n},\mathbf{p}}(\mu)$ is strongly pseudoconvex at all points of $b_{0}D_{n_{0}}^{\mathbf{n},\mathbf{p}}(\mu)$.\\
$(2)$  $D_{n_{0}}^{\mathbf{n},\mathbf{p}}(\mu)$ is weakly pseudoconvex but not strongly pseudoconvex  at any point of $b_{1}D_{n_{0}}^{\mathbf{n},\mathbf{p}}(\mu)$ and is not smooth at any point of $b_{2}D_{n_{0}}^{\mathbf{n},\mathbf{p}}(\mu)$.
\end{Proposition}

\begin{proof}[Proof]
Let
$$\rho(z,w_{(1)},\cdots,w_{(l)}):=\sum\limits_{j=1}^{l}{\left\lVert w_{(j)}\right\rVert}^{2p_{j}}-e^{-\mu{\left\lVert z\right\rVert}^{2}}.$$
Then $\rho$ is a real analytic definition function of $b_{0}D_{n_{0}}^{\mathbf{n},\mathbf{p}}(\mu)$. Fix a point $(z_{0},w_{(1)0},\cdots,w_{(l)0})\in b_{0}D_{n_{0}}^{\mathbf{n},\mathbf{p}}(\mu)$ and let
$T=(\zeta,\eta_{(1)},\cdots,\eta_{(l)})\in T_{(z_{0},w_{(1)0},\cdots,w_{(l)0})}^{1,0}(b_{0}D_{n_{0}}^{\mathbf{n},\mathbf{p}}(\mu)).$
Then by definition, we know that
\begin{equation}\label{eq2.14}
w_{(j)0}\neq 0,\;\; j=1+\epsilon,\cdots,l;
\end{equation}
\begin{equation}\label{eq2.15}
\sum\limits_{k=1}^{l}p_{k}{\left\lVert w_{(k)0}\right\rVert}^{2(p_{k}-1)}\overline{w_{(k)0}}\cdot\eta_{(k)}+\mu e^{-\mu{\left\lVert z_{0}\right\rVert}^{2}}\overline{z_{0}}\cdot\zeta=0;
\end{equation}
\begin{equation}\label{eq2.16}
\sum\limits_{j=1}^{l}{\left\lVert w_{(j)0}\right\rVert}^{2p_{j}}-e^{-\mu{\left\lVert z_{0}\right\rVert}^{2}}=0.
\end{equation}
Thanks to \eqref{eq2.14}, \eqref{eq2.15} and \eqref{eq2.16}, the Levi form of $\rho$ at the point $(z_{{0}},w_{(1)0},\cdots,w_{(l)0})$ can be computed as follows:
\begin{equation*}
\begin{aligned}
&L_{\rho}(T,T)\\
&:=\sum\limits_{i,j=1}^{n_{0}+n_{1}+\cdots+n_{l}}\frac{\partial^2\rho}{\partial T_{i}\partial
\overline{T_{j}}}(z_{0},w_{(1)0},\cdots,w_{(l)0})T_{i}\overline{T_{j}}\\
&=\sum\limits_{k=1}^{l}p_{k}(p_{k}-1){\Vert w_{(k)0}\Vert}^{2(p_{k}-2)}{\vert\overline{w_{(k)0}}\cdot\eta_{(k)}\vert}^{2}
+\sum\limits_{k=1}^{l}p_{k}{\Vert w_{(k)0}\Vert}^{2(p_{k}-1)}{\Vert\eta_{(k)}\Vert}^{2}\\
&\;\;\;+\mu e^{-\mu{\left\lVert z_{0}\right\rVert}^{2}}{\Vert \zeta\Vert}^{2}-\mu^{2}
e^{-\mu{\left\lVert z_{0}\right\rVert}^{2}}{\vert\overline{z_{0}}\cdot\zeta\vert}^{2}\\
\end{aligned}
\end{equation*}
\begin{equation*}
\begin{aligned}
&=\sum\limits_{k=1}^{l}{p_{k}}^{2}{\Vert w_{(k)0}\Vert}^{2(p_{k}-2)}{\vert\overline{w_{(k)0}}\cdot\eta_{(k)}\vert}^{2}
+\mu e^{-\mu{\left\lVert z_{0}\right\rVert}^{2}}{\Vert \zeta\Vert}^{2}-\mu^{2}e^{-\mu{\left\lVert z_{0}\right\rVert}^{2}}{\vert\overline{z_{0}}\cdot\zeta\vert}^{2}\\
&\;\;\;+\sum\limits_{k=1}^{l}p_{k}{\Vert w_{(k)0}\Vert}^{2(p_{k}-2)}({\Vert w_{(k)0}\Vert}^{2}{\Vert\eta_{(k)}\Vert}^{2}-{\vert\overline{w_{(k)0}}\cdot\eta_{(k)}\vert}^{2})\\
\end{aligned}
\end{equation*}
\begin{equation*}
\begin{aligned}
&=(\sum\limits_{k=1}^{l}{\Vert w_{(k)0}\Vert}^{2p_{k}})^{-1}\bigg(\sum\limits_{k=1}^{l}{p_{k}}^{2}{\Vert w_{(k)0}\Vert}^{2(p_{k}-2)}{\vert\overline{w_{(k)0}}\cdot\eta_{(k)}\vert}^{2}\bigg)\bigg(\sum\limits_{k=1}^{l}{\Vert w_{(k)0}\Vert}^{2p_{k}}\bigg)\\
&\;\;\;-(\sum\limits_{k=1}^{l}{\Vert w_{(k)0}\Vert}^{2p_{k}})^{-1}{\bigg|\sum\limits_{k=1}^{l}{p_{k}}{\Vert w_{(k)0}\Vert}^{2(p_{k}-1)}\overline{w_{(k)0}}\cdot\eta_{(k)}\bigg|}^{2}\\
\end{aligned}
\end{equation*}
\begin{equation*}
\begin{aligned}
&\;\;\;+\sum\limits_{k=1}^{l}p_{k}{\Vert w_{(k)0}\Vert}^{2(p_{k}-2)}\bigg({\Vert w_{(k)0}\Vert}^{2}{\Vert\eta_{(k)}\Vert}^{2}-{\vert\overline{w_{(k)0}}\cdot\eta_{(k)}\vert}^{2}\bigg)+\mu e^{-\mu{\left\lVert z_{0}\right\rVert}^{2}}{\Vert \zeta\Vert}^{2}\\
\end{aligned}
\end{equation*}
\begin{equation*}
\begin{aligned}
&=(\sum\limits_{k=1}^{l}{\Vert w_{(k)0}\Vert}^{2p_{k}})^{-1}\bigg[\bigg(\sum\limits_{k=1}^{l}{p_{k}}^{2}{\Vert w_{(k)0}\Vert}^{2(p_{k}-2)}{\vert\overline{w_{(k)0}}\cdot\eta_{(k)}\vert}^{2}\bigg)\bigg(\sum\limits_{k=1}^{l}{\Vert w_{(k)0}\Vert}^{2p_{k}}\bigg)\\
&\;\;\;-{\bigg|\sum\limits_{k=1}^{l}{p_{k}}{\Vert w_{(k)0}\Vert}^{2(p_{k}-1)}\overline{w_{(k)0}}\cdot\eta_{(k)}\bigg|}^{2}\bigg]+
\mu e^{-\mu{\left\lVert z_{0}\right\rVert}^{2}}{\Vert \zeta\Vert}^{2}\\
&\;\;\;+\sum\limits_{k=1}^{l}p_{k}{\Vert w_{(k)0}\Vert}^{2(p_{k}-2)}\bigg({\Vert w_{(k)0}\Vert}^{2}{\Vert\eta_{(k)}\Vert}^{2}-{\vert\overline{w_{(k)0}}\cdot\eta_{(k)}\vert}^{2}\bigg)\geq \mu e^{-\mu{\left\lVert z_{0}\right\rVert}^{2}}{\Vert \zeta\Vert}^{2}\geq 0
\end{aligned}.
\end{equation*}
by the Cauchy-Schwarz inequality, for all $T=(\zeta,\eta_{(1)},\cdots,\eta_{(l)})\in T_{(z_{0},w_{(1)0},\cdots,w_{(l)0})}^{1,0}(b_{0}D_{n_{0}}^{\mathbf{n},\mathbf{p}}(\mu))$. Obviously, if $\zeta\neq0$, then $L_{\rho}(T,T)>0$.

On the other hand, combining with \eqref{eq2.14}, \eqref{eq2.15} and \eqref{eq2.16}, we know that the equality holds if and only if
\begin{equation}
\zeta=0,
\end{equation}
\begin{equation}
{\Vert w_{(k)0}\Vert}^{2}{\Vert\eta_{(k)}\Vert}^{2}-{\vert\overline{w_{(k)0}}\cdot\eta_{(k)}\vert}^{2}=0,
\end{equation}
\begin{equation}\label{eq2.17}
\begin{aligned}
&\bigg[\bigg(\sum\limits_{k=1}^{l}{p_{k}}^{2}{\Vert w_{(k)0}\Vert}^{2(p_{k}-2)}{\vert\overline{w_{(k)0}}\cdot\eta_{(k)}\vert}^{2}\bigg)\bigg(\sum\limits_{k=1}^{l}{\Vert w_{(k)0}\Vert}^{2p_{k}}\bigg)\\
&-{\bigg|\sum\limits_{k=1}^{l}{p_{k}}{\Vert w_{(k)0}\Vert}^{2(p_{k}-1)}\overline{w_{(k)0}}\cdot\eta_{(k)}\bigg|}^{2}\bigg]=0.
\end{aligned}
\end{equation}

Suppose $\zeta=0$, then $T=(\zeta,\eta_{(1)},\cdots,\eta_{(l)})\neq0$ implies that there exists $\eta_{i_{0}}\neq0$. If $L_{\rho}(T,T)=0$  for all $T\neq0\in T_{(z_{0},w_{(1)0},\cdots,w_{(l)0})}^{1,0}(b_{0}D_{n_{0}}^{\mathbf{n},\mathbf{p}}(\mu))$, then by \eqref{eq2.14}, \eqref{eq2.15}, \eqref{eq2.16} and \eqref{eq2.17}, we have $\eta_{k}=0$ ($1\leq k \leq l$). This is a contradiction.

When there exists $j_{0}\geq 1+\epsilon$ such that ${\left\lVert w_{(j_{0})0}\right\rVert}^{2}=0$ and $p_{j_{0}}> 1$, then $(z_{0},w_{(1)0},\cdots,w_{(l)0})\in b_{1}D_{n_{0}}^{\mathbf{n},\mathbf{p}}(\mu)$. Let $T_{0}=(0,\cdots,\eta_{(j_{0})},0,\cdots,0),\;\Vert\eta_{(j_{0})}\Vert\neq 0$. Then $L_{\rho}(T_{0},T_{0})=0$. Hence $D_{n_{0}}^{\mathbf{n},\mathbf{p}}(\mu)$ is weakly pseudoconvex but not strongly pseudoconvex  on any point of $b_{1}D_{n_{0}}^{\mathbf{n},\mathbf{p}}(\mu)$.

It is obvious that $D_{n_{0}}^{\mathbf{n},\mathbf{p}}(\mu)$ is not smooth at any point of $b_{2}D_{n_{0}}^{\mathbf{n},\mathbf{p}}(\mu)$. The proof is completed. \end{proof}

\begin{Lemma}[Tu-Wang \cite{T-W.Math Ann}]\label{Thm1.3}
  Let $\Sigma(\mathbf{n};\mathbf{p})$ and $\Sigma(\mathbf{m};\mathbf{q})$ be two equidimensional generalized pseudoellipsoids, $\mathbf{n},\mathbf{m}\in \mathbb{N}^{l}$, $\mathbf{p},\mathbf{q}\in \mathbb{(R_{+})}^{l}$ (where $p_{k},\;q_{k}\neq1\;for\;2\leq k\leq l$). Let
  $h:\Sigma(\mathbf{n};\mathbf{p})\rightarrow \Sigma(\mathbf{m};\mathbf{q})$ be a biholomorphic linear isomorphism between
  $\Sigma(\mathbf{n};\mathbf{p})$ and $\Sigma(\mathbf{m};\mathbf{q})$. Then there exists a permutation $\sigma\in S_r$ such that
  $n_{\sigma(i)}=m_i,\;p_{\sigma(i)}=q_i$ and
  \begin{eqnarray*}
  \begin{aligned}
  h(\zeta_{1},\cdots,\zeta_{r})=(\zeta_{\sigma(1)},\cdots,\zeta_{\sigma(r)})
  \begin{bmatrix}
  U_1 &     &       &    \\
    & U_2 &       &    \\
    &     &\ddots &    \\
    &     &       & U_r\\
  \end{bmatrix}
  \end{aligned},
  \end{eqnarray*}
  where $U_i$ is a unitary transformation of $\mathbb{C}^{m_i}(m_i=n_{\sigma(i)})$ for $1\leq i\leq r$.
  \end{Lemma}

Define
$$V_{1}:=\{(z,w_{(1)},\cdots,w_{(l)}) \in \mathbb{C}^{n_{0}}\times \mathbb{C}^{n_{1}}\times\cdots\times\mathbb{C}^{n_{l}}:  \;
w_{(1)}=0,\cdots,w_{(l)}=0 \}\;(\cong \mathbb{C}^{n_{0}}),$$
$$V_{2}:=\{(z,w_{(1)},\cdots,w_{(l)})\in\mathbb{C}^{m_{0}}\times \mathbb{C}^{m_{1}}
\times\cdots\times\mathbb{C}^{m_{l}}:\; w_{(1)}=0,\cdots,w_{(l)}=0\}\; (\cong \mathbb{C}^{m_{0}}).$$
Then we have the following lemma.

\begin{Lemma}\label{Lemma 3.1} Suppose $D_{n_{0}}^{\mathbf{n},\mathbf{p}}(\mu)$
and $D_{m_{0}}^{\mathbf{m},\mathbf{q}}(\nu)$
are two equidimensional generalized Fock-Bargmann-Hartogs domains, $f\;:D_{n_{0}}^{\mathbf{n},\mathbf{p}}(\mu)\rightarrow D_{m_{0}}^{\mathbf{m},\mathbf{q}}(\nu)$
 is a biholomorphic mapping. Then we have $f(V_{1})\subseteq V_{2}$ and $f|_{V_{1}}\;:V_{1}\rightarrow V_{2}$
 is biholomorphic. Consequently  $n_{0}=m_{0}$.
\end{Lemma}
\begin{proof}[Proof] Let
$f(z,0)=(f_{1}(z),f_{2}(z))$, then we get
$\sum\limits_{i=1}^{l}{\Vert
f_{2i}\Vert}^{2q_{i}}<e^{-\nu{\left\lVert
f_{1}(z)\right\rVert}^{2}}\leq 1.$
Then we obtain that the bounded entire mapping
$f_{2i}(z)$ on $\mathbb{C}^{n_{0}}$ is constant  $(1\leq i\leq l)$ by Liouville's Theorem. Since
$f(z)$ is biholomorphic,
$f_{1}(z)$ is an unbounded function. Hence there exist
$\{z_{k}\}$ such that
$f_{1}(z_{k})\rightarrow\infty$ as $k\rightarrow\infty$. It implies $f_{2}(z)\equiv0$. This proves $f(V_{1})\subseteq V_{2}$.
Similarly, by making the same argument for $f^{-1}$,
we have $f^{-1}(V_{2})\subseteq V_{1}$. Namely, $f|_{V_{1}}\;:V_{1}\rightarrow V_{2}$ is biholomorphic.
Hence $n_{0}=m_{0}$.
\end{proof}

Now we give the proof of Theorem \ref{Thm1}.

\begin{proof}[The proof of Theorem \ref{Thm1}]
Let $f(0,0)=(a,b)$ (thus $b=0$ by Lemma \ref{Lemma 3.1}) and define
$$\phi(z,w_{(1)},\cdots,w_{(l)}):=(z-a,w_{(1)}(e^{2\nu{<z,a>}-\nu{\left\lVert
a\right\rVert}^{2}})^{\frac{1}{2q_{1}}},\cdots,w_{(l)}(e^{2\nu{<z,a>}-\nu{\left\lVert a\right\rVert}^{2}})^
{\frac{1}{2q_{l}}}).$$
Obviously, $\phi\in\mathrm{Aut}(D_{m_{0}}^{\mathbf{m},\mathbf{q}}(\nu))$ and $\phi\circ f(0,0)=(0,0)$.
Then $\phi\circ f$ is linear by  Theorem \ref{Thm2.2.1}. We describe $\phi\circ f$ as follows:
\begin{equation*}
\phi\circ f(z,w)=(z,w)
\left(\begin{array}{cc}
A&B\\
C&D\\
\end{array}\right)=(zA+wC,zB+wD).
\end{equation*}
According to Lemma \ref{Lemma 3.1}, we have $f(z,0)=(f_{1}(z),0)$. Thus $B=0$. Since $g:=\phi\circ f
$ is biholomorphic,  $A$ and $D$ are invertible matrices. We write $g(z,w)$ as follows:
\begin{equation*}
g(z,w)=(z,w)
\left(\begin{array}{cc}
A&0\\
C&D\\
\end{array}\right)=(z,w_{(1)},\cdots,w_{(l)})
\left(\begin{array}{ccccc}
A&0&\cdots&0\\
C_{11}&D_{11}&\cdots&D_{1l}\\
\vdots&\vdots&\ddots&\cdots\\
C_{l1}&D_{l1}&\cdots&D_{ll}\\
\end{array}\right),
\end{equation*}
which implies that
\begin{equation*}
g^{-1}(z,w)=(z,w)
\left(\begin{array}{cc}
A^{-1}&0\\
-D^{-1}CA^{-1}&D^{-1}\\
\end{array}\right)=(z,w_{(1)},\cdots,w_{(l)})\left(\begin{array}{ccccc}
A^{-1}&0&\cdots&0\\
E_{11}&G_{11}&\cdots&G_{1l}\\
\vdots&\vdots&\ddots&\vdots\\
E_{l1}&G_{l1}&\cdots&G_{ll}\\
\end{array}\right).
\end{equation*}
Set $\Sigma(\mathbf{n};\mathbf{p})=\{(w_{(1)},\cdots,w_{(l)})\in\mathbb{C}^{n_{1}}\times\cdots
\times\mathbb{C}^{n_{l}}:\sum\limits_{j=1}^{l}{\Vert w_{(j)}\Vert}^{2p_{j}}<1\}.$ Then, if $\sum\limits_{j=1}^{l}{\Vert w_{(j)}\Vert}^{2p_{j}}<e^{-\mu{\left\lVert 0\right\rVert}^{2}}=1$, we obtain
$$\sum\limits_{j=1}^{l}{\Vert w_{(1)}D_{1j}+\cdots+w_{(l)}D_{lj}\Vert}^{2q_{j}}<e^{-\nu{\left\lVert
wC\right\rVert}^{2}}<1$$ and
if $\sum\limits_{j=1}^{l}{\Vert w_{(j)}\Vert}^{2q_{j}}<e^{-\nu{\left\lVert 0\right\rVert}^{2}}=1$,
 we have $$\sum\limits_{j=1}^{l}{\Vert w_{(1)}G_{1j}+\cdots+w_{(l)}G_{lj}\Vert}^{2p_{j}}
 <e^{-\mu{\left\lVert w(-D^{-1}CA^{-1})\right\rVert}^{2}}<1.$$\\
Therefore, we conclude that the mapping
$g_{2}(w):\Sigma(\mathbf{n};\mathbf{p})
\rightarrow \Sigma(\mathbf{m};\mathbf{q})$ given by
\begin{equation*}
g_{2}(w_{(1)},\cdots,w_{(l)})=wD=(w_{(1)},\cdots,w_{(l)})\left(\begin{array}{ccc}
D_{11}&\cdots&D_{1l}\\
\vdots&\ddots&\vdots\\
D_{l1}&\cdots&D_{ll}\\
\end{array}\right)
\end{equation*}
is a biholomorphic linear mapping. By Lemma \ref{Thm1.3}, $g_{2}$ can be expressed in the form:
\begin{equation*}
g_{2}(w_{(1)},\cdots,w_{(l)})=(w_{(\sigma(1))},\cdots,w_{(\sigma(l))})\left(\begin{array}{cccc}
\Gamma_{1}&\null&\null&\null\\
\null&\Gamma_{2}&\null&\null\\
\null&\null&\ddots&\null\\
\null&\null&\null&\Gamma_{l}\\
\end{array}\right),
\end{equation*}
where $\sigma \in S_{l}$ is a permutation with
$n_{\sigma(j)}=m_{j}$, $p_{\sigma(j)}=q_{j} \; (j=1,\cdots,l)$ and  $\Gamma_{i}\in \mathcal{U}(m_{i})\;(1\leq i\leq l)$.
Hence $g$ can be rewritten as follows:
\begin{equation*}
g(z,w)=(z,w)
\left(\begin{array}{cc}
A&0\\
C&D\\
\end{array}\right)=(z,w_{(\sigma(1))},\cdots,w_{(\sigma(l))})
\left(\begin{array}{ccccc}
A&\null&\null&\null\\
C_{\sigma(1)1}&\Gamma_{1}&\null&\null\\
C_{\sigma(2)1}&\null&\Gamma_{2}&\null\\
\vdots&\null&\null&\ddots&\null\\
C_{\sigma(l)1}&\null&\null&\null&\Gamma_{l}\\
\end{array}\right).
\end{equation*}

Next we prove that $C=0.$  The linearity of $g$ yields that
$g(bD_{n_{0}}^{\mathbf{n},\mathbf{p}})=bD_{m_{0}}^{\mathbf{m},\mathbf{q}}.$
Let
$(0,w)=(0,0,\cdots,w_{(j)},0,\cdots,0)\in bD_{n_{0}}^{\mathbf{n},\mathbf{p}}$,
namely, ${\Vert w_{(j)}\Vert}^{2}=(e^{-\mu{\left\lVert
0\right\rVert}^{2}})^{\frac{1}{p_{j}}}=1$. As $\Gamma_{j}$ $(1\leq j\leq l)$ are unitary matrices, moreover,
 assuming $\sigma(i_{0})=j$, we have
$${\Vert w_{(j)}\Vert }^{2p_{j}}={\Vert w_{(\sigma(i_{0}))}\Gamma_{i_{0}}\Vert}^{2q_{i_{0}}}=e^{-\nu{\left\lVert
w_{(\sigma(i_{0}))}C_{\sigma(i_{0})1}\right\rVert}^{2}}=1.$$
This implies
$w_{(j)}C_{j1}=0$ for all ${\Vert w_{(j)}\Vert}^{2}=1$. So $C_{j1}=0\; (1\leq j\leq l).$
Thus we have
\begin{equation*}
g(z,w_{(1)},\cdots,w_{(l)})=(z,w_{(\sigma(1))},\cdots,w_{(\sigma(l))})
\left(\begin{array}{ccccc}
A&\null&\null&\null&\null\\
\null&\Gamma_{1}&\null&\null\null\\
\null&\null&\Gamma_{2}\null&\null\\
\null&\null&\null&\ddots&\null\\
\null&\null&\null&\null&\Gamma_{l}\\
\end{array}\right).
\end{equation*}

Lastly, we show $\sqrt{\frac{\nu}{\mu}}A\in \mathcal{U}(n)$ $(n:=n_0=m_0)$. For $z\in\mathbb{C}^{n_{0}}$, take $(w_{(1)},\cdots,w_{(l)})$ such that
$e^{-\mu{\left\lVert z\right\rVert}^{2}}=\sum\limits_{j=1}^{l}{\Vert w_{(j)}\Vert}^{2p_{j}}.$
By $g(bD_{n_{0}}^{\mathbf{n},\mathbf{p}})=bD_{m_{0}}^{\mathbf{m},\mathbf{q}}$, we have $\sum\limits_{j=1}^{l}{\Vert w_{(\sigma(j))}\Gamma_{j}\Vert}^{2q_{j}}=e^{-\mu{\left\lVert zA\right\rVert}^{2}}.$
Since $ \Gamma_{j}(j=1,\cdots,l)$ are unitary matrices,
we get $$e^{-\mu{\left\lVert z\right\rVert}^{2}}=
\sum\limits_{j=1}^{l}{\Vert w_{(\sigma(j))}\Vert}^{2p_{\sigma(j)}}=\sum\limits_{j=1}^{l}
{\Vert w_{(\sigma(j))}\Gamma_{j}\Vert}^{2q_{j}}=e^{-\nu{\left\lVert zA\right\rVert}^{2}}.$$
Therefore, $\nu\Vert zA\Vert^{2}=\mu{\left\lVert z\right\rVert}^{2}$ $(z\in \mathbb{C}^{n})$. Then we get $\sqrt{\frac{\nu}{\mu}}A\in \mathcal{U}(n)$, and the proof is completed.
\end{proof}

\begin{proof}[The proof of Corollary \ref{Coro.1.5}]
In fact, the significance of the above $\phi$ is just to ensure $\phi\circ f(0)=0$. Then the proof of  Theorem \ref{Thm1} implies that Corollary \ref{Coro.1.5} is obvious.
\end{proof}

\begin{proof}[The proof of the Theorem \ref{Thm.1.6}] Obviously, $\varphi_{A}, \varphi_{D}
\textrm{ and }\varphi_{a}$ are biholomorphic self-mappings of
$D_{n_{0}}^{\mathbf{n},\mathbf{p}}(\mu)$.
On the other hand, for $\varphi\in \mathrm{Aut}(D_{n_{0}}^{\mathbf{n},\mathbf{p}}(\mu))$, we
assume $\varphi(0,0)=(a,b)$ (then $b=0$ by Lemma \ref{Lemma 3.1}). Hence $\varphi_{-a}\circ\varphi$ preserves the origin. Then by Corollary \ref{Coro.1.5}, we obtain $\varphi_{-a}\circ\varphi=\varphi_{D}\circ\varphi_{A}$
for some $\varphi_{A}, \varphi_{D}$.
 Hence $\varphi=\varphi_{a}\circ\varphi_{D}\circ\varphi_{A}$, and the proof is complete.
\end{proof}

\begin{proof}[The proof of Theorem \ref{Thm1.2}]
Let $f$ be a proper holomorphic mapping between two equidimensional generalized Fock-Bargmann-Hartogs domains $D_{n_{0}}^{\mathbf{n},\mathbf{p}}(\mu)$ and $D_{m_{0}}^{\mathbf{m},\mathbf{q}}(\nu)$. Then by Th. \ref{Thm2.3.2},
 $f$  extends holomorphically to a neighborhood $\Omega$ of $\overline{D_{n_{0}}^{\mathbf{n},\mathbf{p}}(\mu)}$ with
 $$f({bD_{n_{0}}^{\mathbf{n},\mathbf{p}}(\mu)})\subset bD_{m_{0}}^{\mathbf{m},\mathbf{q}}(\nu).$$
Then by Proposition \ref{pro.2.4.1} and Lemma 1.3 in Pin\v{c}uk \cite{Pin}, we have
\begin{equation}\label{eq3.15}
f(M\cap b_{0}D_{n_{0}}^{\mathbf{n},\mathbf{p}}(\mu))\subset b_{1}D_{m_{0}}^{\mathbf{m},\mathbf{q}}(\nu)\cup b_{2}D_{m_{0}}^{\mathbf{m},\mathbf{q}}(\nu)
\end{equation}
 where
$M:=\big\{z\in \Omega,\mathrm{det}(\frac{\partial f_{i}}{\partial z_{j}})=0\big\}$  is the zero locus of the complex Jacobian of the holomorphic mapping $f$  on $\Omega$.

If $M\cap bD_{n_{0}}^{\mathbf{n},\mathbf{p}}(\mu)\neq \emptyset$, then, from min $\{n_{1+\epsilon},n_{2},\cdots,n_{l}\}\geq2$, we have $M\cap b_{0}D_{n_{0}}^{\mathbf{n},\mathbf{p}}(\mu)\neq \emptyset$. Take an irreducible component $M'$ of $M$ with $M'\cap b_{0}D_{n_{0}}^{\mathbf{n},\mathbf{p}}(\mu)\neq \emptyset$. Then the intersection $E_{M'}$ of $M'$ with $b_{0}D_{n_{0}}^{\mathbf{n},\mathbf{p}}(\mu)$ is a real
analytic submanifold of dimensional $2(n_{0}+n_{1}+\cdots+n_{l})-3$ on a dense, open subset of $E_{M'}$. By \eqref{eq3.15}, we have $f(E_{M'})\subset b_{1}D_{m_{0}}^{\mathbf{m},\mathbf{q}}(\nu)\cup b_{2}D_{m_{0}}^{\mathbf{m},\mathbf{q}}(\nu)$. Hence
\begin{equation}\label{eq3.16}
f(M'\cap D_{n_{0}}^{\mathbf{n},\mathbf{p}}(\mu))\subset \bigcup\limits_{j=1+\delta}^{l}Pr_{i}(D_{m_{0}}^{\mathbf{m},\mathbf{q}}(\nu)),
\end{equation}
where $Pr_{i}(D_{m_{0}}^{\mathbf{m},\mathbf{q}}(\nu)):=\{(z,w_{(1)},\cdots,w_{(l)})\in D_{m_{0}}^{\mathbf{m},\mathbf{q}}(\nu),\;\Vert w_{(i)}\Vert=0\}$ ($1+\delta\leq i\leq l$), by the uniqueness theorem. Since $\mathrm{codim}M'=1$, $\mathrm{codim}[\bigcup\limits_{j=1+\delta}^{l}Pr_{i}(D_{m_{0}}^{\mathbf{m},\mathbf{q}}(\nu))]\geq \textrm{min}\{m_{1+\delta}
,\cdots,m_{l},m_{1}+\cdots+m_{l}\}\geq2$ and $f: D_{n_{0}}^{\mathbf{n},\mathbf{p}}(\mu)\rightarrow D_{m_{0}}^{\mathbf{m},\mathbf{q}}(\nu)$ is proper, this is contradiction with \eqref{eq3.16}. Thus we have $M\cap bD_{n_{0}}^{\mathbf{n},\mathbf{p}}(\mu)=\emptyset$.

Let $S:=M\cap D_{n_{0}}^{\mathbf{n},\mathbf{p}}(\mu)$. Hence we have
$$S\subset D_{n_{0}}^{\mathbf{n},\mathbf{p}}(\mu),\;\;\overline{S}\cap bD_{n_{0}}^{\mathbf{n},\mathbf{p}}(\mu)=\emptyset.$$

If $S\neq\emptyset$, then $S$ is a complex analytic set in $\mathbb{C}^{n_{0}+ n_{1}+\cdots+
 n_{l}}$ also. For any $(z,w)\in S$, we have
${\vert w_{ln_{l}}}\vert^{2p_{l}}\leq \sum\limits_{j=1}^{l}{\Vert w_{(j)}\Vert}^{2p_{j}}\leq e^{-\mu{\left\lVert z\right\rVert}^{2}}
\leq 1.$
Thus
\begin{equation}\label{equ19}
{\vert w_{ln_{l}}}\vert^{2}\leq 1\leq 1+\Vert(z,w')\Vert,
\end{equation}
where $w=(w',w_{ln_{l}})$. Then $S$ is an algebraic set of $\mathbb{C}^{n_{0}+n_{1}+\cdots+
 n_{l}}$ by \S 7.4 Th. 3 of Chirka \cite{Chirka}.

Suppose $S_{1}$ is an irreducible component of $S$. Let $\overline{S_{1}}$ be the closure of $S_{1}$ in $ \mathbb{P}^{n_{0}+ n_{1}+\cdots+
 n_{l}}$. Then by \S 7.2 Prop. 2 of Chirka \cite{Chirka}, $\overline{S_{1}}$ is a projective algebraic set and $\dim\overline{S_{1}}=n_{0}+n_{1}+\cdots+
n_{l}-1$.
Let $[\xi,z,w]$ be the homogeneous coordinate in $\mathbb{P}^{n_{0}+ n_{1}+\cdots+
 n_{l}}$, we embed $\mathbb{C}^{n_{0}+ n_{1}+\cdots+
 n_{l}}$ into $\mathbb{P}^{n_{0}+ n_{1}+\cdots+
n_{l}}$ as the affine piece $U_{0}=\{[\xi,z,w]\in \mathbb{P}^{n_{0}+ n_{1}+\cdots+
n_{l}},\;\xi\neq 0\}$ by $(z,w)\hookrightarrow [1,z,w]$. Then we have $$D_{n_{0}}^{\mathbf{n},\mathbf{p}}(\mu)\cap U_{0}=\bigg\{[\xi,z,w],\;\xi\neq0,\;\sum\limits_{j=1}^{l}\frac{{\Vert w_{(j)}\Vert}^{2p_{j}}}{{\vert \xi\vert}^{2p_{j}}}<e^{-\mu\frac{{\Vert z \Vert}^{2}}{{\vert \xi\vert}^{2}}}\bigg\}.$$

Let $H=\{\xi=0\}\subset\mathbb{P}^{n_{0}+ n_{1}+\cdots+
 n_{l}}$. Consider another affine piece $U_{1}=\{[\xi,z,w]\in \mathbb{P}^{n_{0}+n_{1}+\cdots+
 n_{l}},\;z_{1}\neq 0\}$ with affine coordinate $(\zeta,t,s)=(\zeta,t_{2},\cdots,t_{n_{0}},s_{(1)},\cdots,s_{(l)})$. Let $t'=(1,t_{2},\cdots,t_{n_{0}})$.
 Since $\frac{{\Vert w_{(j)}\Vert}^{2p_{j}}}{{\vert\xi\vert}^{2p_{j}}}=\frac{{\Vert w_{(j)}\Vert}^{2p_{j}}}{{\vert z_{1}\vert}^{2p_{j}}}\frac{{\vert z_{1}\vert}^{2p_{j}}}{{\vert\xi\vert}^{2p_{j}}}=\frac{{\Vert s_{(j)}\Vert}^{2p_{j}}}{{\vert\zeta\vert}^{2p_{j}}}$ and $e^{-\mu\frac{{\Vert z \Vert}^{2}}{{\vert \xi\vert}^{2}}}=e^{-\mu\frac{{\Vert z \Vert}^{2}}{{\vert z_{1}\vert}^{2}}\frac{{\vert z_{1}\vert}^{2}}{{\vert\xi\vert}^{2}}}=e^{-\mu\frac{1+{\vert t_{2}\vert}^{2}+\cdots+{\vert t_{n_{0}}\vert}^{2}}{{\vert \zeta\vert}^{2}}}$, we obtain
\begin{equation}\label{eq3.18}
\begin{aligned}
&D_{n_{0}}^{\mathbf{n},\mathbf{p}}(\mu)\cap U_{0}\cap U_{1}\\
=&\bigg\{(\zeta,t_{2},\cdots,t_{n_{0}},s_{(1)},\cdots,s_{(l)})\in \mathbb{C}^{n_{0}+n_{1}+\cdots+
 n_{l}},\;\sum\limits_{j=1}^{l}\frac{{\Vert s_{(j)}\Vert}^{2p_{j}}}{{\vert \zeta\vert}^{2p_{j}}}<e^{-\mu\frac{{\Vert t' \Vert}^{2}}{{\vert \zeta\vert}^{2}}}\bigg\}.\\
\end{aligned}
\end{equation}
Let $S'=\overline{S_{1}}\cap U_{1}$ and $H_{1}=H\cap U_{1}=\{\zeta=0\}$ (note $\xi=\frac{\zeta}{z_{1}}$). For every $u\in S'\cap H_{1}$, there exists a sequence of points $\{u_{k}\}\subset \overline{S_{1}}\cap ((U_{0}\cap U_{1})\backslash H_{1})$ such that
$u_{k}\rightarrow u\;(k\rightarrow\infty)$, The formula \eqref{eq3.18} implies
\begin{equation}\label{eq3.19}
{\Vert s_{(j)}(u_{k})\Vert}^{2p_{j}}\leq {\vert \zeta(u_{k})\vert}^{2p_{j}}e^{-\mu\frac{{\Vert t' \Vert}^{2}}{{\vert \zeta(u_{k})\vert}^{2}}},\;1\leq j\leq l.
\end{equation}
Since $u\in H_{1}$, that means $\zeta(u)=0$ and $\zeta{(u_{k})}\rightarrow 0\;(k\rightarrow\infty)$. Therefore we have ${\Vert s_{(j)}(u)\Vert}^{2p_{j}}\leq0\;(1\leq j\leq l)$ as $k\rightarrow\infty$. Hence
\begin{equation}
S'\cap H_{1}\subset \bigg\{\zeta=0,\;s_{(1)}=\cdots=s_{(l)}=0\bigg\}.$$
 Then $\dim(S'\cap H_{1})\leq n_{0}-1$. Shafarevich \cite{Shafa} \S6.2 Th. 6 implies
$$n_{0}-1\geq
 \dim(S'\cap H_{1})\geq \dim S'+\dim H_{1}-n_{0}-n_{1}-\cdots-
 n_{l}\geq \dim S'-1.
\end{equation}
This means $\dim S'\leq n_{0}$, and thus $n_{0}+ n_{1}+\cdots+
 n_{l}-1=\dim S'\leq n_{0}$. Therefore, we get $ n_{1}+\cdots+
n_{l}\leq 1$, a contradiction with assumption min $ \{n_{1+\epsilon},n_{2},\cdots,n_{l},n_{1}+\cdots+
n_{l}\}\geq 2$.

Therefore, $S=\emptyset$ and thus $f$ is unbranched. Since the generalized Fock-Bargmann-Hartogs domain is simply connected,  $f:D_{n_{0}}^{\mathbf{n},\mathbf{p}}(\mu)\rightarrow D_{m_{0}}^{\mathbf{m},\mathbf{q}}(\nu)$ is a biholomorphism. The proof is completed. \end{proof}

 \noindent\textbf{Acknowledgments}
We sincerely thank the referees, who read the paper very carefully and gave many useful suggestions. E. Bi was supported by the Natural Science Foundation of Shandong Province, China (No.ZR2018BA015), and Z. Tu was supported by
the National Natural Science Foundation of China (No.11671306).

\addcontentsline{toc}{section}{References}
\phantomsection
\renewcommand\refname{References}
\small{

\end{document}